\documentclass[a4paper]{amsart}
\usepackage[english]{babel}

\usepackage[latin1]{inputenc}
\usepackage{verbatim}
\usepackage[arrow, matrix, curve]{xy}
\usepackage{amsmath,amssymb,amsfonts,amsthm,amsopn}
\usepackage{nicefrac}
\usepackage{mathtools}
\usepackage{enumerate}
\usepackage{a4wide}

\title{Two remarks on Narkiewicz's property (P)}
\date{\today}

\author{Lukas Pottmeyer}
\address{Fakult\"at f\"ur Mathematik, Universit\"at Duisburg-Essen, 45117 Essen}
\email{lukas.pottmeyer@unidue.de}

\DeclareMathOperator{\Gal}{Gal}
\DeclareMathOperator{\id}{id}

\DeclareMathOperator{\Per}{Per}

\newtheorem{theorem}{Theorem}[section]

\newtheorem{corollary}[theorem]{Corollary}

\newtheorem{construction}[theorem]{Construction}
\newtheorem{proposition}[theorem]{Proposition}

\newtheorem{lemma}[theorem]{Lemma}

\theoremstyle{definition}
\newtheorem*{definition}{Definition}

\theoremstyle{remark}
\newtheorem{example}[theorem]{Example}
\newtheorem{remark}[theorem]{Remark}

\newcommand{\Symm}[1]{\mathbb{Q}_{#1,\rm sym}}

\makeindex

\begin{document}

\begin{abstract}
Due to Narkiewicz a field $F$ has property (P) if for no polynomial $f\in F[x]$ of degree at least two there is an infinite $f$-invariant subset of $F$. We present a new example of an algebraic extension of $\mathbb{Q}$ satisfying (P). This is the first example in which we can find points of arbitrarily small positive Weil-height. Moreover, we study the possibility of property (P) for the field generated by all symmetric Galois extensions of $\mathbb{Q}$. In particular we prove that there are no infinite backward orbits of non linear polynomials in this field. 
\end{abstract}

\subjclass[2010]{37P05, 37P35, 11R04}
\keywords{polynomial dynamics, periodic points, property (P), heights}
\maketitle

\section{Property (P)}

In this paper we deal with a definition due to Narkiewicz.

\begin{definition}
A field $F$ has property (P), if there is no infinite subset $X\subseteq F$, such that $f(X)=X$ for some polynomial $f\in F[x]$ with $\deg(f)\geq 2$.
\end{definition}

A similar definition, due to Liardet, replaces \emph{polynomial} by \emph{rational function} in the above definition. In this case the field is said to have property (R). It is not known whether property (P) is equivalent to property (R). This is one of several open questions on property (P) (see \cite{Na63}, \cite{Na71}, and \cite[Sections IX and X]{Na}). One of these questions asks for a constructive classification of all fields with property (P). With a view to the known examples, see Examples \ref{ex} and Corollary \ref{cor:weird} below, this seems to be the most difficult question on property (P).

This property is obviously closely tied to the theory of dynamical systems. Hence, we recall the basic definitions from this discipline. Let $f\in\overline{\mathbb{Q}}(x)$ be a rational function. With $f^n$ we denote the $n$th iterate of $f$, with the usual convention $f^0=x$. A point $\alpha \in \overline{\mathbb{Q}}$ is called a periodic point of $f$, if for some integer $n\geq 1$ we have $f^n(\alpha)=\alpha$. The smallest $n$ with this property is called the exact period of $\alpha$. The set of all periodic points of $f$ of exact period $n$ is denoted by $\Per_n(f)$ and we set $\Per(f)=\cup_{n\geq 1}\Per_n(f)$.

\begin{remark}\label{rmk:(P)}
It follows immediately that a field $F$ has property (P) if and only if every $f\in F[x]$, with $\deg(f)\geq 2$, satisfies:
\begin{itemize}
\item[(P1)] there is no infinite sequence $\alpha_0,\alpha_1,\ldots$ of pairwise distinct elements $\alpha_i \in F$, such that $f(\alpha_i)=\alpha_{i-1}$ for all $i\in\mathbb{N}$, and
\item[(P2)] there are only finitely many periodic points of $f$ in $F$.
\end{itemize} 
\end{remark}

In this paper we are only interested in fields with property (P) lying inside a fixed algebraic closure $\overline{\mathbb{Q}}$ of the rational numbers. In this case, height functions are an important tool for proving that a certain $F\subseteq \overline{\mathbb{Q}}$ has property (P). Denote by $h$ the (absolute logarithmic Weil-)height on $\overline{\mathbb{Q}}$. An algebraic extension $F$ of $\mathbb{Q}$ is said to have the Northcott property (N) if for every $T\in \mathbb{R}$ the set $\{\alpha \in F \vert h(\alpha)\leq T\}$ is finite. It is easy to see that property (N) implies property (P), we refer to \cite{CW} and \cite{DZ2} for a proof and additional results.
There are only few examples known of subfields of $\overline{\mathbb{Q}}$ with property (P).

\begin{example}\label{ex}
Let $K$ denote an arbitrary number field. The following fields satisfy property (P):
\begin{enumerate}[(I)]
\item $F=\cup_{i\in\mathbb{N}} K_i$, where $K=K_0\subseteq K_1 \subseteq K_2 \subseteq \ldots$ is a nested sequence of number fields, such that
\[
\inf_{K_i \subsetneq M \subseteq K_{i+1}} (N_{K_i /\mathbb{Q}}(D_{M/K_i}))^{\nicefrac{1}{[M:K_0][M:K_i]}} \longrightarrow \infty,
\]
where $D_{M/K_i}$ denotes the relative discriminant of the extension $M/K_i$ and $N_{K_i/\mathbb{Q}}$ is the norm associated to $K_i/\mathbb{Q}$ (see \cite{Wi});
\item $F\subseteq K(\mu)$, where $\mu$ is the set of roots of unity, and $L(\mu)\cap F$ and $L(\{\zeta + \zeta^{-1}\vert \zeta \in \mu\})\cap F$ are finite for all linear polynomials $L\in \overline{\mathbb{Q}}[x]$ (see \cite{DZ1});
\item $F/K$ Galois, such that infinitely many local degrees of $F$ are finite (see \cite{Po}).
\end{enumerate}
\end{example}

A nice example of fields from part (III) are the fields which are generated over $\mathbb{Q}$ by elements of bounded degree. See \cite{Po} for more information.

\begin{remark}
The fields from (I) satisfy property (N). Other fields with property (N) are constructed in \cite[Theorems 1 and 2]{BZ}. Note that the set of examples from \cite[Theorem 2]{BZ} is non-empty as shown in \cite{CF}, but they are already covered by the fields in (III).

Both classes (II) and (III) contain fields which do not satisfy property (N) (see \cite{DZ2} for the fields in (II), and \cite{CF} for the fields in (III)). However, all fields from Example \ref{ex} satisfy a gap principle for the height (also known as the Bogomolov property (B)). This is, for every $F$ from Example \ref{ex} there is a positive constant $c_F$ such that every $\alpha \in F$ satisfies either $h(\alpha)=0$ or $h(\alpha)\geq c_f$. This follows from \cite{BZ} for the fields in (III), from \cite{Wi} for the fields in (I), and from \cite{AZ} for the fields in (II). 
\end{remark}

\subsection{The construction of Kubota and Liardet}

To the best of my knowledge there is only one further class of examples of fields with property (P), due to Kubota and Liardet \cite{KL}:

\begin{construction}\label{const:KL}
Since $\overline{\mathbb{Q}}[x]$ is countable, we enumerate the polynomials of degree $\geq 2$ as $f_1,f_2,\ldots$ Moreover, we choose any infinite sequence $p_1,p_2,\ldots$ of distinct primes. Lastly, we set $E_n = \cup_{j\leq n} \Per(f_j)$. Now we construct a sequence
\[
K_0\subseteq K_1 \subseteq K_2 \subseteq \ldots
\] 
such that $K_0$ is an arbitrary number field and for each $n\geq 1$ we have $[K_n:K_{n-1}]=p_n$ and $K_n \cap E_n=K_{n-1}\cap E_n$. Then $F=\cup_{n\geq 1} K_n$ has property (P).
\end{construction}

Note that the condition $K_n \cup E_n=K_{n-1}\cup E_n$ leads to (P2) for $F$, and the condition $[K_n:K_{n-1}]=p_n$ leads to (P1) for $F$. Dvornicich and Zannier \cite[p. 535]{DZ1} noticed that this \emph{construction is rather indirect, and it seems difficult to detect other properties of the field so obtained}. In this note we will use a very similar construction which can produce explicit fields satisfying property (P) not contained in Example \ref{ex}. In these explicit fields we can detect other properties. In particular, we can produce a field which satisfies property (P), but - in contrast to all fields from Example \ref{ex} - does contain points of arbitrarily small positive height. The precise formulation of this result is as follows.

Let $K$ be a number field, $L/K$ a finite extension, and $v$ a non-archimedean valuation on $K$. If $w\mid v$ is an extension of $v$ to $L$, then we denote by $e_{w\mid v}$, resp. $f_{w \mid v}$, the ramification index, resp. the inertia degree, of $L_w / K_v$. If $F/K$ is any algebraic extension, then we say that $F$ has finite inertia degree over $v$, if there exists a constant $C_v$ such that for every finite subextension $F/L/K$ there is a $w\mid v$ on $L$ such that $f_{w\mid v}\leq C_v$. Equivalently, $F/K$ has finite inertia degree over $v$, if there is an extension $w$ of $v$ to $F$ such that the residue field of $F_w$ is finite.

\begin{proposition}\label{prop:P2}
Let $F \subseteq \overline{\mathbb{Q}}$ be a field. If for infinitely many prime numbers $p$ the extension $F/\mathbb{Q}$ has finite inertia degree over $p$, then $F$ satisfies (P2).
\end{proposition}

The proof, which follows quite immediately from a result of Morton and Silverman \cite[Corollary B]{MS} (see also \cite[Corollary 2]{Pe} for a related result), can be found in Section \ref{sec:prop}. 

Note that Proposition \ref{prop:P2} reproves (P2) for all fields in (III). In what follows $\sqrt[n]{a}$ denotes any fixed root of the polynomial $x^n-a$. The proof of the next corollary is the content of Section \ref{sec:cor}.

\begin{corollary}\label{cor:weird}
Let $(p_n)_{n\geq 1}$ be a sequence of pairwise distinct prime numbers. Define $\alpha_0=1$ and $\alpha_{n}=\sqrt[p_{n}]{p_{n}\alpha_{n-1}}$ for all $n\geq 1$. Then
\[
F=\mathbb{Q}(\alpha_1,\alpha_2,\ldots)
\] 
has property (P). Moreover, $h(\alpha_n)\rightarrow 0$ as $n$ tends to infinity.
\end{corollary}

Note that the degree of $\mathbb{Q}(\alpha_1,\ldots,\alpha_n)/\mathbb{Q}(\alpha_1,\ldots,\alpha_{n-1})$ is equal to the prime number $p_n$. As in Construction \ref{const:KL}, this will guarantee that the field $F$ satisfies (P1).

\subsection{The field $\mathbb{Q}^{\rm sym}$}
Another hypothetical possibility to prove (P2) for a field $F$ lies in the structure of the Galois groups of the fields generated by periodic points. Let $K$ be a number field, $f\in K[x]$, and $\alpha \in \Per_n(f)$. Moreover, we denote by $K_{\alpha}$ the Galois closure of $K(\alpha)/K$.

We introduce the following hypothesis on $f$:
\begin{equation}\label{eq:hyp}
\begin{aligned}
 &\text{\it for all but finitely many } \alpha \in \Per(f) \text{ \it the group } \Gal(K_{\alpha}/K)  \\ & \text{\it has a normal abelian subgroup of exponent } \geq 25.
\end{aligned}
\end{equation}

Morton and Patel \cite[Theorem 7.4]{MP} proved that for a generic polynomial $f\in\mathbb{Q}[x]$ of degree $\geq 2$ any $\alpha \in \Per(f)$ is a Galois conjugate of $f(\alpha)$, and that $\Gal(K_{\alpha}/K)$ has an abelian normal subgroup of exponent $n$. In particular, it should be the regular case that a polynomial satisfies hypothesis \eqref{eq:hyp}. We will discuss this in more detail in Section \ref{sec:periodicsym}. 
Let $\mathbb{Q}^{\rm sym}$ be the compositum of all finite Galois extensions of $\mathbb{Q}$ with Galois group isomorphic to some symmetric group $S_n$. We will give some support for the possibility that $\mathbb{Q}^{\rm sym}$ satisfies property (P).

\begin{theorem}\label{thm:Qsym}
The field $\mathbb{Q}^{\rm sym}$ satisfies (P1). Moreover, every polynomial $f\in\mathbb{Q}^{\rm sym}$ that satisfies hypothesis \eqref{eq:hyp} for some number field $K$ has at most finitely many periodic points in $\mathbb{Q}^{\rm sym}$.
\end{theorem}  

In \cite{Fe} it was shown that there exists a pseudo algebraically closed (PAC) field which satisfies property (N), and hence (P). Since $\mathbb{Q}^{\rm sym}$ is PAC (see \cite[Theorem 18.10.4]{FJ}), Theorem \ref{thm:Qsym} gives a potential explicit example of a PAC field satisfying (P).

Section \ref{sec:P1} contains the proof of (P1) for $\mathbb{Q}^{\rm sym}$. The remaining statement of Theorem \ref{thm:Qsym} is proven in Section \ref{sec:periodicsym}. We colse this introduction by noting that property (P) can actually be defined on any algebraic variety, when we replace polynomials by endomorphisms and the field $F$ by $F$-rational points. In this setting, property (R) deals with endomorphisms of $\mathbb{P}^1$. Narkiewicz also introduced the case (in modern language) of endomorphisms of $\mathbb{P}^N$ for arbitrary $N$ and called this property (SP). The case of property (P) for elliptic curves has recently been studied in \cite{MSha}.

\section{Ramification and proof of Proposition \ref{prop:P2}}\label{sec:prop}

Let $K$ be a number field with ring of integers $\mathcal{O}_K$. We use the usual one-to-one correspondence between non-zero prime ideals in $\mathcal{O}_K$ and non-archimedean valuations on $K$. Let $v$ be the valuation corresponding to the prime ideal $\mathfrak{P}$ in $\mathcal{O}_K$, and let $p$ be the rational prime with $v\mid p$. The norm of $\mathfrak{P}$ is given by
\begin{equation}\label{eq:norm}
N(\mathfrak{P})=\vert \nicefrac{\mathcal{O}_K}{\mathfrak{P}}\vert =  p^{f_{v\mid p}}
\end{equation}
A polynomial $f(x)=a_dx^d + a_{d-1}x^{d-1} + \ldots +a_0\in K[x]$ has good reduction at $v$ if $v(a_i) \leq 1$ for all $i\in\{0,\ldots,d-1\}$, and $v(a_d)=1$. For the necessary facts from valuation and ramification theory, we refer to the first three chapters of \cite{Ne} (or the English translation thereof), and for information on reduction of polynomials and rational maps, we refer to \cite[Section 2.5]{Si}. 

We need the following obvious facts:
\begin{itemize}
\item A polynomial in $K[x]$ has good reduction at all but finitely many valuations of $K$. 
\item  Let $L/K$ be a finite extension with an extension $w\mid v$ of non-archimedean valuations. If $f\in K[x]$ has good reduction at $v$, then $f$ considered as a polynomial in $L[x]$ has good reduction at $w$.
\end{itemize}

Proposition \ref{prop:P2} will follow immediately from the following result (see \cite[Corollary 2.26]{Si}), where we have applied Equation \eqref{eq:norm}.

\begin{lemma}\label{lem:nice}
Let $K$ be a number field and let $v,w$ be non-archimedean valuations on $K$. Assume further that $v \mid p$ and $w\mid q$ for distinct rational primes $p$ and $q$. If $f \in K[x]$ has good reduction at $v$ and $w$, and if $\alpha \in \Per_n(f)\cap K$, then $n \leq (p^{2f_{v\mid p}}-1)(q^{2f_{w\mid p}}-1)$.
\end{lemma} 

\begin{proof}[Proof of Proposition \ref{prop:P2}]
Let $F$ be an algebraic extension of $\mathbb{Q}$ with finite inertia degree over infinitely many primes. Let $f\in F[x]$ and $\alpha \in \Per_n(f)\cap F$ be arbitrary. There is some number field $K\subseteq F$ with $f \in K[x]$ and $\alpha \in K$. Moreover, there are necessarily two distinct primes $p$ and $q$, such that 
\begin{itemize}
\item $f$ has good reduction at all valuations on $K$ lying above $p$ and $q$, and
\item $F$ has finite inertia degree over $p$ and $q$.
\end{itemize}
Now, we fix $v\mid p$ and $w\mid q$ in $K$ such that $f_{v\mid p}$ and $f_{w\mid q}$ are bounded from above by constants $C_p$ and $C_q$ only depending on $p$, $q$ and $F$. It follows from Lemma \ref{lem:nice} that $n \leq (p^{2C_p}-1)(q^{2C_q}-1)$. Hence, we have bounded the possible period size of an element from $F\cap \Per(f)$. In particular, there are at most finitely many periodic points for $f$ in $F$. This proves Proposition \ref{prop:P2}.
\end{proof}

The next result can be seen as a slightly more effective version of Construction \ref{const:KL}.

\begin{corollary}\label{cor:tower}
Let $K_0$ be a number field and $(p_n)_{n\geq 1}$ a sequence of distinct prime numbers. Let $K_0\subseteq K_1 \subseteq K_2 \subseteq \ldots$ be a nested sequence of number fields, 
such that $K_n/K_{n-1}$ is totally ramified at all valuations lying above $p_1,\ldots,p_n$. Then $F=\cup_{n\geq 0}K_n$ satisfies (P2). If in addition $[K_n:K_{n-1}]=p_n$ for all $n \geq 1$, then $F$ satisfies property (P).
\end{corollary} 

\begin{proof}
Let $n\geq 1$ be arbitrary, and let $v_1^{n-1},\ldots,v_k^{n-1}$ be the extensions of $p_n$ to $K_{n-1}$. That the extension $K_n/K_{n-1}$ is totally ramified above all $v_1^{n-1},\ldots,v_k^{n-1}$ implies that for all $v_i^{n-1}$, $i\in \{1,\ldots,k\}$, there is precisely one extension $v_i^{n}$ to $K_n$, and this satisfies $f_{v_i^{n}\mid v_i^{n-1}}=1$. Inductively it follows that in any $K_N$ with $N\geq n$ the extensions of $p_n$ are given by some $v_1^{N},\ldots,v_k^{N}$, with $f_{v_i^{N}\mid v_i^{N-1}}=1$ for all $i\in\{1,\ldots,k\}$. Since the inertia degree is multiplicative, we have $f_{v_i^{N}\mid p_n}=f_{v_i^{n-1}\mid p_n}$. In particular, $F=\cup_{n\geq 0}K_n$ has finite inertia degree above all primes $p_1,p_2,\ldots$. Hence, by Proposition \ref{prop:P2} the field $F$ satisfies (P2).

Now we assume that we have $[K_n:K_{n-1}]=p_n$ for all $n \geq 1$. Then (P1) for $F$ follows precisely as in the work of Kubota and Liardet \cite{KL}. For completeness we recall the simple argument here. The assumption on the degree implies that there is no proper intermediate field in the extension $K_{n}/K_{n-1}$. Equivalently, the field $K_{n}$ is the smallest proper field extension of $K_{n-1}$ lying in $F$. 
Let $f\in F[x]$ be of degree $d\geq 2$, and let $\beta_0 \in F$ be arbitrary. Fix any integer $n$ such that $\beta_0 \in K_n$, $f\in K_n[x]$, and such that $p_n >d$. Construct (if possible) pairwise distinct elements $\beta_1,\ldots,\beta_k\in K_n$ such that $f(\beta_i)=\beta_{i-1}$ for all $i\in\{1,\ldots,k\}$. Since any number field satisfies property (P) -- and hence (P1) -- such a sequence must necessarily be finite. Assume that $k$ is maximal, and that there is some $\beta_{k+1} \in F$, with $f(\beta_{k+1})=\beta_k$. Then $1\neq [K_n(\beta_{k+1}):K_n]\leq d < p_n$, which gives a contradiction. Hence, $F$ satisfies (P1), which implies that $F$ satisfies property (P).
\end{proof}

\section{Proof of Corollary \ref{cor:weird}}\label{sec:cor}

Let $(p_n)_{n\geq 1}$ be a sequence of pairwise distinct primes. We define
\[
\alpha_0=1 \quad \text{ and } \quad \alpha_n = \sqrt[p_n]{p_n\alpha_{n-1}} \text{ for } n\geq 1.
\]
Moreover, we set
\[
F=\mathbb{Q}(\alpha_1,\alpha_2,\ldots).
\]
We consider $F$ as the union of number fields. Define $K_0=\mathbb{Q}$ and $K_n=K_{n-1}(\alpha_n)$ for all $n\geq 1$, then $F=\cup_{n\geq 0}K_n$. We have
\begin{equation}\label{eq:alphpower}
\alpha_n^{p_1\cdots p_n}=(p_1)\cdot (p_2^{p_1})\cdots (p_n^{p_1\cdots p_{n-1}}).
\end{equation}
Hence, $\alpha_n$ is the root of the $p_1$-Eisenstein polynomial $x^{p_1\cdots p_n}-\alpha_n^{p_1\cdots p_n} \in \mathbb{Z}[x]$. This implies that $[K_n:\mathbb{Q}]=p_1\cdots p_n$, which in turn yields 
\begin{equation}\label{eq:primedegree}
[K_n:K_{n-1}]=p_n \qquad \text{ for all } n\geq 1.
\end{equation}
Note that all primes $p_k$, with $k> n$, are unramified in $K_n$. We need one further result on ramification in radical extensions.

\begin{lemma}\label{lem:radicalramification}
Let $K$ be a number field, $\alpha\in K$, and $p$ a prime number. Assume that the fractional ideal $\alpha \mathcal{O}_K$ decomposes as $\mathfrak{p}_1^{a_1}\cdots \mathfrak{p}_r^{a_r}$ for certain prime ideals $\mathfrak{p}_1,\ldots,\mathfrak{p}_r$ in the ring of integers $\mathcal{O}_K$. If $p\nmid a_i$, then $\mathfrak{p}_i$ is totally ramified in the extension $K(\sqrt[p]{\alpha})/K$.
\end{lemma}
\begin{proof}
The short proof can be found in \cite[p. 132]{MV}.
\end{proof}

\begin{proof}[Proof of Corollary \ref{cor:weird}]
In order to prove that $F$ has property (P), thanks to Corollary \ref{cor:tower} and \eqref{eq:primedegree}, we are left to prove that the valuations (or equivalently: prime ideals) above $p_1,\ldots,p_n$ are totally ramified in $K_n/K_{n-1}$. We prove this by induction, where  the induction base is the well-known statement that $p_1$ is totally ramified in $\mathbb{Q}(\sqrt[p_1]{p_1})$. Assume that the claim is correct for some $n\geq 1$. Let $i\in\{1,\ldots,n+1\}$ be arbitrary, and let $\mathfrak{p}$ be any prime ideal in $\mathcal{O}_{K_{i-1}}$ above $p_i$. Since $p_i$ is unramified in $K_{i-1}$, we have $e_{\mathfrak{p}\mid p_i}=1$.  Moreover, by our induction hypotheses all prime ideals above $\mathfrak{p}$ are totally ramified in $K_n$. Let $\mathfrak{P}$ be the unique prime ideal in $\mathcal{O}_{K_n}$ above $\mathfrak{p}$. Then $e_{\mathfrak{P}\mid \mathfrak{p}} = [K_n:K_{i-1}]=p_i\cdots p_n$. 
Combining this with $e_{\mathfrak{p}\mid p_i}=1$, gives 
\begin{equation}\label{eq:e}
e_{\mathfrak{P}\mid p_i}=p_i\cdots p_n.
\end{equation}
Since, by \eqref{eq:alphpower}, we know $$p_n\alpha_{n-1} \mathcal{O}_{K_n} \mid (p_n \cdot p_1 p_2^{p_1}\cdots p_{n-1}^{p_1\cdots p_{n-2}})\mathcal{O}_{K_n},$$ the exponent of $\mathfrak{P}$ in $p_n\alpha_{n-1} \mathcal{O}_{K_n}$ must be a divisor of 
\[
p_1\cdots p_{i-1}\cdot e_{\mathfrak{P}\mid p_i} \overset{\eqref{eq:e}}{=} p_1\cdots p_{n}.
\]
Hence, $p_{n+1}$ is not a divisor of this exponent and by Lemma \ref{lem:radicalramification} it follows, that $\mathfrak{P}$ is totally ramified in $K_{n+1}$. This concludes the induction and applying Corollary \ref{cor:tower} yields that $F$ has property (P2), and hence property (P).

It remains to prove that the height of the elements $\alpha_n$ tends to zero as $n$ tends to infinity. 
We have $h(\alpha_1)=h(\sqrt[p_1]{p_1})=\frac{1}{p_1}\log(p_1) <1$, and 
\begin{equation}\label{eq:heightbound}
h(\alpha_{n+1})=\frac{1}{p_{n+1}}h(p_{n+1}\alpha_n)\leq \frac{1}{p_{n+1}}(\log(p_{n+1})+h(\alpha_n))
\end{equation}
for all $n \geq 1$. A trivial induction shows that $h(\alpha_n)<1$ for all $n\geq 1$, but this implies by \eqref{eq:heightbound} that $h(\alpha_n)$ indeed tends to zero as $n$ tends to infinity. 
\end{proof}

\section{Some facts about $\mathbb{Q}^{\rm sym}$}\label{sec:Qsym}

\begin{definition}
For $k\in\mathbb{N}$ we define the field $\Symm{k}$ to be the compositum of all number fields $K$, such that $K/\mathbb{Q}$ is Galois with Galois group isomorphic to some symmetric group $S_n$, with $n\leq k$. Then $\mathbb{Q}^{\rm sym}=\cup_{k\in\mathbb{N}} \Symm{k}$.
\end{definition}

\begin{remark}\label{rmk:Symmk}
It is $\Symm{k}$ generated by elements of degree bounded from above by $k!$. Hence, for every $k$, the field $\Symm{k}$ and all its finite extensions have property (P), by Example \ref{ex} (III).
\end{remark}

\begin{lemma}\label{lem:singel generator}
Let $k\geq 4$ be a rational integer and let $K/\mathbb{Q}$ be a finite Galois extension. If $\mathbb{Q}(\alpha)/\mathbb{Q}$ is Galois, with Galois group isomorphic to $S_n$, then the group $\Gal(\Symm{k}K(\alpha)/\Symm{k}K)$ is either trivial or isomorphic to the alternating group $A_n$, for $n\geq k$.
\end{lemma}
\begin{proof}
Obviously, we have $\alpha \in \Symm{k}$, for $n\leq k$. Hence, we assume for the rest of this proof $n>k\geq 4$. It is
\[
\Gal(\Symm{k}K(\alpha)/\Symm{k}K) \cong \Gal(\mathbb{Q}(\alpha)/\mathbb{Q}(\alpha)\cap\Symm{k}K)=: H \subseteq S_n.
\]
Moreover, $\mathbb{Q}(\alpha)\cap\Symm{k}K/\mathbb{Q}$ is Galois, and hence $H$ is normal in $S_n$. Since $n> 4$, it follows $H\in \{\id, A_n, S_n\}$, and we are left to show $H\neq S_n$. 

The field $\Symm{k}$ contains all quadratic field extensions of $\mathbb{Q}$. In particular, $\Symm{k}K$ contains $\mathbb{Q}(\alpha)^{A_n}$, the field fixed by $A_n \subseteq \Gal(\mathbb{Q}(\alpha)/\mathbb{Q})$. Hence, $\mathbb{Q}(\alpha)\cap\Symm{k}K\neq \mathbb{Q}$ and $H\neq S_n$.
\end{proof}

\begin{proposition}\label{prop:several generators}
Let $K/\mathbb{Q}$ be a finite Galois extension and let $F$ be a subfield of $K\mathbb{Q}^{\rm sym}$ such that $F/\Symm{k}K$ is also a finite Galois extension, for some $k\geq 4$. Then
\[
\Gal(F/\Symm{k}K)\cong A_{n_1}\times \ldots \times A_{n_r},
\]
with $r\in\mathbb{N}_0$ and $n_1,\ldots,n_r >k$.
\end{proposition}
\begin{proof}
By definition of $\mathbb{Q}^{\rm sym}$ there are $\alpha_1,\ldots,\alpha_{r'}$ such that $\mathbb{Q}(\alpha_i)/\mathbb{Q}$ is Galois with group $S_{n_i}$ for all $i\in\{1,\ldots,r'\}$, and such that $F\subseteq \Symm{k}K(\alpha_1,\ldots,\alpha_{r'})$.

We may assume that none of the $\alpha_1,\ldots,\alpha_{r'}$ can be omitted. This implies $n_i>k$ for all $i\in\{1,\ldots,r'\}$. By Lemma \ref{lem:singel generator} it is $\Gal(\Symm{k}K(\alpha_i)/\Symm{k}K) \cong A_{n_i}$ for all $i\in\{1,\ldots,r'\}$. 

Let $i\in\{1,\ldots ,r'-1\}$ be arbitrary. It is $$(\Symm{k}K(\alpha_1,\ldots,\alpha_{i})\cap \Symm{k}K(\alpha_{i+1}))/\Symm{k}K$$ a Galois extension, and the Galois group is a proper (since $\alpha_{i+1}$ cannot be omitted) normal subgroup of $\Gal(\Symm{k}K(\alpha_{i+1})/\Symm{k}K)\cong A_{n_{i+1}}$ -- hence it is trivial and $\Symm{k}K(\alpha_1,\ldots,\alpha_{i})\cap \Symm{k}K(\alpha_{i+1})=\Symm{k}K$. Basic Galois theory gives the result
\[
\Gal(\Symm{k}K(\alpha_1,\ldots,\alpha_{r'})/\Symm{k}K) \cong A_{n_1}\times \ldots\times A_{n_{r'}}.
\]
Therefore, $\Gal(F/\Symm{k}K) \trianglelefteq A_{n_1}\times \ldots\times A_{n_{r'}}$. By Goursat's lemma, the normal subgroup $\Gal(F/\Symm{k}K)$ must be a direct product of some of the factors of $A_{n_1}\times \ldots\times A_{n_{r'}}$. Hence, after a possible renumbering, we have $\Gal(F/\Symm{k}K)\cong A_{n_1}\times \ldots\times A_{n_{r}}$ for some $r\leq r'$.
\end{proof}

\begin{corollary}\label{cor:abelian in Symm4}
Let $K/\mathbb{Q}$ be a finite Galois extension, and let $K'\subseteq \mathbb{Q}^{\rm sym}$ be a number field such that $K'K/K$ is Galois and abelian. Then it is $K'\subseteq \Symm{4}K$.
\end{corollary}
\begin{proof}
The group $\Gal(K'K\Symm{4}/\Symm{4}K)$ is abelian. Hence, Proposition \ref{prop:several generators} implies that this group is trivial, proving the claim.
\end{proof}

\section{$\mathbb{Q}^{\rm sym}$ satisfies (P1)}\label{sec:P1}

We will prove that $\mathbb{Q}^{\rm sym}$ satisfies (P1). 
Hence, we pick an arbitrary $f\in \mathbb{Q}^{\rm sym}[x]$ of degree $d \geq 2$. 
Let $\alpha_0, \alpha_1,\ldots \in \overline{\mathbb{Q}}$ be pairwise distinct elements such that $f(\alpha_i)=\alpha_{i-1}$ for all $i\in\mathbb{N}$, and such that $\alpha_0 \in \mathbb{Q}^{\rm sym}$. We will prove that this sequence is not contained in $\mathbb{Q}^{\rm sym}$. To this end, let $K/\mathbb{Q}$ be a finite Galois extension, with $K\subset \mathbb{Q}^{\rm sym}$, $f\in K[x]$, and $\alpha_0 \in K$. Set $k=\max\{\deg(f),4\}$, and consider the field $\Symm{k}K$. Since $\Symm{k}K$ has property (P), there are only finitely many elements from $\{\alpha_i\}_{i\in\mathbb{N}_0}$ in $\Symm{k}K$. Let $n\in\mathbb{N}_0$ be maximal such that $\alpha_n \in \Symm{k}K$. We will show that $\alpha_{n+1}\notin \mathbb{Q}^{\rm sym}$. 

Assume for the sake of contradiction that $\alpha_{n+1}\in \mathbb{Q}^{\rm sym}$. Let $F$ be the Galois closure of $\Symm{k}K(\alpha_{n+1})$ over $\Symm{k}K$. Then $F\subseteq \mathbb{Q}^{\rm sym}$, and by Proposition \ref{prop:several generators} it is
\begin{equation}\label{eq:1}
\Gal(F/\Symm{k}K)\cong A_{n_1}\times \ldots \times A_{n_r},
\end{equation}
with $n_1,\ldots,n_r > \deg(f)$. However, the minimal polynomial of $\alpha_{n+1}$ over $\Symm{k}K$ divides $f(x) - \alpha_n$. Hence $\Gal(F/\Symm{k}K)$ is a subgroup of $S_{\deg(f)}$, contradicting \eqref{eq:1}.

It follows that $\alpha_{n+1} \notin \mathbb{Q}^{\rm sym}$, and hence the sequence $\alpha_0,\alpha_1,\ldots$ cannot be fully contained in $\mathbb{Q}^{\rm sym}$, proving the proposition.

\section{Periodic points in $\mathbb{Q}^{\rm sym}$}\label{sec:periodicsym}

Let $K$ be a number field, $f\in K[x]$ of degree $\geq 2$, and $\alpha \in \Per_n(f)$. For any $\sigma\in\Gal(\overline{\mathbb{Q}}/K)$, we have 
\[
f^{(n)}(\sigma(\alpha))=\sigma(f^{(n)}(\alpha))=\sigma(\alpha),
\]
and hence $\sigma (\alpha)\in \Per_n(f)$. Obviously, it is also $f(\alpha)\in \Per_n(f)$. It follows that the set of Galois conjugates of $\alpha$ over $K$ is contained in a finite union of periodic orbits of length $n$. Say
\begin{equation}\label{eq:subset}
\Gal(\overline{\mathbb{Q}}/K)\cdot \alpha \subseteq O_1\cup \ldots \cup O_r,
\end{equation}
where $O_i=\{f^{(k)}(\alpha_i)\vert 0\leq k \leq n-1\}$ for some $\alpha_i \in \Per_n(f)$. Denote by $K_{\alpha}$ the Galois closure of $K(\alpha)/K$. 
If $\tau \in \Gal(K_{\alpha}/K)$ maps $\alpha_i$ to some $f^{(k)}(\alpha_j)$, for some integer $k$, then $\tau(f^{(\ell)}(\alpha_i))=f^{(\ell)}(\tau(\alpha_i))\in O_j$. This means that any element in $\Gal(K_{\alpha}/K)$ permutes the orbits $O_1,\ldots,O_r$. Set 
\begin{equation}\label{eq:N}
N_{\alpha,K}=\{\sigma \in \Gal(K_{\alpha}/K) \vert \sigma(O_i)=O_i ~\forall ~i\in\{1,\ldots,r\}\}.
\end{equation}
Then $N_{\alpha,K}$ is a normal subgroup of $\Gal(K_{\alpha}/K)$, which injects into $\left(\nicefrac{\mathbb{Z}}{n\mathbb{Z}}\right)^r$. This follows precisely as  in the proof of \cite[Theorem 4.1]{MP}.

Morton and Patel \cite[Theorem 7.4]{MP} proved that in the generic situation, there is a $\sigma \in \Gal(K_{\alpha}/K)$ such that for all $\beta \in \Gal(K_{\alpha}/K)\cdot \alpha$ we have $\sigma(\beta)=f(\beta)$. In particular, this $\sigma$ is an element in $N_{\alpha,K}$ of order $n$. Hence, for an arbitrarily given $f\in K[x]$ and an $\alpha \in \Per_n(f)$, one expects that the exponent of $N_{\alpha,K}$ grows linearly in $n$. The $25$ in our hypothesis \eqref{eq:hyp} is  particularly chosen for our purposes. We will need below that the exponent of a normal abelian subgroup of $\Gal(K_{\alpha}/K)$ is greater than the exponent of $\Gal(\mathbb{Q}_{4,{\rm sym}}/\mathbb{Q})$, which is precisely $4!=24$.

\begin{remark}
Let $f(x)=x^d$ be a power map, and $\zeta\in\Per_n(f)$. Then $\zeta$ is a primitive $d^{n}-1$-st root of unity. Now, for every number field $K /\mathbb{Q}$, at least one of the elements $\zeta^d,\ldots,\zeta^{d^{[K:\mathbb{Q}]}}$ is a Galois conjugate of $\zeta$ over $K$. This means that one of the maps $f,f^{(2)},\ldots,f^{(d^{[K:\mathbb{Q}]})}$ induces a $K$-automorphism on $K(\zeta)$. In particular, for any $n\geq 25[K:\mathbb{Q}]$ and any $\alpha \in \Per_n(f)$, the exponent of $N_{\alpha,K}$ is greater than $25$.

The very same argument is also valid if $f$ is a Chebyshev polynomial. Hence, every power map and every Chebyshev polynomial satisfies hypothesis \eqref{eq:hyp}.
\end{remark}

\begin{proposition}
Let $f\in\mathbb{Q}^{\rm sym}[x]$ be of degree $\geq 2$, such that $f$ satisfies hypothesis \eqref{eq:hyp}, then $\Per(f)\cap \mathbb{Q}^{\rm sym}$ is a finite set. 
\end{proposition}
\begin{proof}
Let $f$ be as in the statement, and let $K\subseteq \mathbb{Q}^{\rm sym}$ be a number field with $f\in K[x]$. 
Let $\alpha$ be an element from $\mathbb{Q}^{\rm sym} \cap \Per(f)$, such that $\Gal(K_{\alpha}/K)$ contains an abelian normal subgroup $N$ with exponent $\geq 25$. Denote with $K_{\alpha}^N$ the fixed field of $N$. Then $K_\alpha/K_{\alpha}^N$ and $K_\alpha^N/K$ are Galois extensions, where $\Gal(K_\alpha/K_\alpha^N)\cong N$ is abelian.

By Corollary \ref{cor:abelian in Symm4}, it is $K_\alpha\subseteq \Symm{4}K_\alpha^N$. Therefore, 
\[
\Gal(K_\alpha/K_\alpha^N)\cong N \trianglelefteq \Gal(\Symm{4}K_\alpha^N/K_\alpha^N).
\]
However, the exponent of $\Gal(\Symm{4}K_\alpha^N/K_\alpha^N)$ is $4!$, contradicting the choice of $N$. If $f$ satisfies hypothesis \eqref{eq:hyp}, then we can conclude that $f$ contains at most finitely many periodic points in $\mathbb{Q}^{\rm sym}$, proving the claim.
\end{proof}

Together with Section \ref{sec:P1} this proves Theorem \ref{thm:Qsym}.

\end{document}